\documentclass{article}
\usepackage{amssymb,latexsym,amsfonts,amsmath,amsthm}
\usepackage{graphics,graphicx,epstopdf}
\usepackage{enumitem}
\usepackage[colorlinks=true,linkcolor=black,citecolor=blue,linktocpage=true]{hyperref}

\newtheorem{theorem}{Theorem}[section]
\newtheorem{corollary}[theorem]{Corollary}
\newtheorem{lemma}[theorem]{Lemma}
\newtheorem{proposition}[theorem]{Proposition}
\newtheorem{definition}{Definition}[section]

\newcommand{\N}{\mathbb{N}}

\newcommand{\spn}{\textnormal{span}}

\newcommand{\dual}[1]{ {#1}^{\ast} }
\newcommand{\ddual}[1]{ {#1}^{\ast\ast} }

\newcommand{\jacobsonradical}[1]{ \texttt{rad}\left({#1}\right) }

\newcommand{\ap}[1]{ \textnormal{ap}\left({#1}\right) }
\newcommand{\wap}[1]{ \textnormal{wap}\left({#1}\right) }
\newcommand{\cc}[1]{ \textnormal{cc}\left({#1}\right) }
\newcommand{\rcc}[1]{ \textnormal{rcc}\left({#1}\right) }
\newcommand{\lcc}[1]{ \textnormal{lcc}\left({#1}\right) }
\newcommand{\wpL}[1]{ \textnormal{wpL}\left({#1}\right) }
\newcommand{\kerwpL}[1]{ W }

\newcommand{\ptensor}[2]{ {#1}\hat{\otimes}_{\pi}{#2} }
\newcommand{\itensor}[2]{ {#1}\hat{\otimes}_{\epsilon}{#2} }

\newcommand{\hull}[1]{\mathbf{h}({#1})}

\begin{document}
\title{A Note on Reflexive Banach Algebras}
\author{Onur Oktay}
\maketitle

\begin{abstract}
We inspect the properties of reflexive Banach algebras, which are related to the pointwise products of its weakly null sequences. 
\end{abstract}


\section{Introduction}

We had studied the properties of the space $\wpL{A}$ for a Banach algebra $A$ in a previous article \cite{Oktay21}. When $A$ is reflexive as a Banach space, $\wpL{A}$ is comprised of almost periodic functionals. This simple observation in \cite{Oktay21} led us to decompose a reflexive amenable Banach algebra as a direct sum of two ideals, one of which is finite dimensional and the other admits no nonzero almost periodic functionals. Our main tool is the set $N$ of all weak limit points of the pointwise products of the weakly null sequences. One of the results (Corollary~\ref{cor:conjecture}) we obtain is an equivalent statement of the conjecture by Gale, Ransford, and White on reflexive amenable Banach algebras \cite{GRW92}. 

In Section~\ref{sec:kerwpL}, we study the properties of the primitive ideals and the maximal left ideals that contain $N$, as well as their implications for reflexive amenable Banach algebras. Inspired by Corollary~\ref{cor:conjecture}, in Section~\ref{sec:wpL0} we provide some properties of reflexive unital Banach algebras, which are spanned by $N$.


\section{Preliminaries and Notation}\label{sec:prelim}

Let $A$ be a Banach algebra and $E$ be a Banach left $A$-module. 
The multiplication operator $L_a:E\to E$ is defined by $L_ax = ax$.
Moreover, one defines
\begin{eqnarray*}
fa(x) = f(ax) = xf(a)
\end{eqnarray*}
for all $a\in A$, $x\in E$ and $f\in\dual{E}$.
%
The maps $T_f:A\to\dual{E}$ and $S_f:E\to\dual{A}$ are defined by $T_fa = fa$ and $S_fx = xf$.

When $E$ is a Banach $A$-bimodule, the right multiplication operator $R_a:E\to E$ is defined by $R_ax = xa$. Further,
\begin{eqnarray*}
af(x) = f(xa) = fx(a)
\end{eqnarray*}
for all $a\in A$, $x\in E$ and $f\in\dual{E}$.
The maps $\sigma_f:A\to\dual{E}$ and $\tau_f:E\to\dual{A}$ are defined by $\sigma_fa = af$ and $\tau_fx = fx$. Clearly, $T_f = \tau_f$ and $S_f=\sigma_f$ when $E=A$.

\vspace{3mm}

The center of an $A$-bimodule $E$ is
\begin{eqnarray*}
\mathcal{Z}(A,E) = \{\texttt{m}\in E : a\texttt{m}=\texttt{m}a\hspace{3mm}\forall a\in A \}
\end{eqnarray*}
If $E=E_1\oplus E_2$ as a direct sum of $A$-bimodules, then 
\begin{eqnarray}\label{eq:center-directsum}
\mathcal{Z}(A,E) = \mathcal{Z}(A,E_1) \oplus \mathcal{Z}(A,E_2)
\end{eqnarray}
as $\mathcal{Z}(A)$-modules. In fact, for any $\texttt{M}\in\mathcal{Z}(A,E)$ there exist unique $\texttt{m}_i\in E_i$ ($i=1,2$) such that $\texttt{M} = \texttt{m}_1 + \texttt{m}_2$. 
Thus, for all $a\in A$
$$0 = a\texttt{M}-\texttt{M}a = (a\texttt{m}_1-\texttt{m}_1a) + (a\texttt{m}_2-\texttt{m}_2a)$$
where, obviously, $(a\texttt{m}_i-\texttt{m}_ia)\in E_i$ ($i=1,2$). Thus, $a\texttt{m}_i=\texttt{m}_ia$ ($i=1,2$).

\vspace{3mm}

We follow the notation and terminology in \cite{Ryan02} for the tensor products. For two Banach spaces $X$ and $Y$, $\ptensor{X}{Y}$ and $\itensor{X}{Y}$ denote the projective and injective tensor products, respectively. 
$\itensor{X}{\dual{Y}}$ is isomorphic to the space of approximable operators $\mathcal{A}(X,Y)$, and $\dual{(\ptensor{X}{Y})}$ is isomorphic to the space of all bounded linear operators $\mathcal{B}(X,\dual{Y})$.
If $X$ or $Y$ has the metric approximation property (AP), then the canonical map $\ptensor{X}{Y}\to\dual{(\itensor{\dual{X}}{\dual{Y}})}$ is an isometric embedding \cite[Theorem 4.14]{Ryan02}.
If $X$ and $Y$ are reflexive (thus have the Radon-Nikodym property) and have AP, then $\ptensor{X}{Y} = \dual{(\itensor{\dual{X}}{\dual{Y}})}$ \cite[Theorem 5.33]{Ryan02}. Thus
\begin{eqnarray}\label{eq:directsumtensor}
\ddual{(\ptensor{X}{Y})} 
= (\ptensor{X}{Y}) \oplus (\itensor{\dual{X}}{\dual{Y}})^{\perp}
\end{eqnarray}
as Banach spaces, where $(\itensor{\dual{X}}{\dual{Y}})^{\perp}$ is the set of all $\texttt{M}\in\ddual{(\ptensor{X}{Y})}$ such that $\texttt{M}(\itensor{\dual{X}}{\dual{Y}})=\{0\}$.
When $A$ is a Banach algebra, $\ptensor{A}{A}$ and $\itensor{A}{A}$ are Banach $A$-bimodules, e.g., the appendix in \cite{Runde02}. 

For a Banach algebra $A$, $\jacobsonradical{A}$ denotes the Jacobson radical of $A$. For a subset $S\subseteq A$, its hull $\hull{S}$ is the set of all primitive ideals of $A$ that contain $S$. If $I$ is the ideal generated by $S$, then $\hull{S} = \hull{I}$. We refer the reader to  \cite{BonsallDuncan73,Palmer1,Rickart60} for an extensive treatment of the general theory of Banach algebras.

\begin{lemma}\normalfont\label{lemma:wpL}
Let $A$ be a Banach algebra, $E$ be a Banach left $A$-module, and $f\in\dual{E}$. Then, the following are equivalent.
\begin{enumerate}
\item[{\it i.}] If $(a_n)$ in $A$ and $(x_n)$ in $E$ are weakly null sequences, then $f(a_nx_n)\to 0$.
\item[{\it ii.}] If $(a_n)$ is a weakly null (resp. weakly Cauchy) sequence in $A$ and $(x_n)$ is a weakly Cauchy (resp. weakly null) sequence in $E$, then $f(a_nx_n)\to 0$.
\item[{\it iii.}] $T_f$ maps weakly precompact sets onto L-sets.
\item[{\it iii'.}] $S_f$ maps weakly precompact sets onto L-sets.
\end{enumerate}
If $E$ is a Banach $A$-bimodule, then the following are also equivalent.
\begin{enumerate}
\item[{\it iv.}] If $(a_n)$ in $A$ and $(x_n)$ in $E$ are weakly null sequences, then $f(x_na_n)\to 0$.
\item[{\it v.}] If $(a_n)$ is a weakly null (resp. weakly Cauchy) sequence in $A$ and $(x_n)$ is a weakly Cauchy (resp. weakly null) sequence in $E$, then $f(x_na_n)\to 0$.
\item[{\it vi.}] $\tau_f$ maps weakly precompact sets onto L-sets.
\item[{\it vi'.}] $\sigma_f$ maps weakly precompact sets onto L-sets.
\end{enumerate}
\end{lemma}
The proof of Lemma~\ref{lemma:wpL} is verbatim similar to the proof of \cite[Lemma~3.1]{Oktay21}. 

\begin{definition}\label{def:wpL-AE}\normalfont
Let $A$ be a Banach algebra and $E$ be a Banach left $A$-module.
\begin{enumerate}[label=\alph*.]
\item $\ap{A,E}$ (resp. $\wap{A,E}$) is the set of all $f\in\dual{E}$ such that $T_f$ is compact (resp. weakly compact),
\item $\wpL{A,E}$ is the set of all $f\in\dual{E}$ such that $T_f$ maps weakly precompact sets onto L-sets,
\item $\lcc{A,E}$ (resp. $\rcc{A,E}$) is the set of all $f\in\dual{E}$ such that $T_f$ (resp. $S_f$) is completely continuous,
\end{enumerate}
and $\cc{A,E} = \lcc{A,E}\cap\rcc{A,E}$. 
When $A=E$, we drop $E$ and simply write $\wpL{A}$, $\lcc{A}$, etc.
\end{definition}
Clearly, $\ap{A,E}\subseteq \cc{A,E} \subseteq \lcc{A,E}\cup\rcc{A,E} \subseteq\wpL{A,E}$.
If both $A$ and $E$ are reflexive as Banach spaces, then $\ap{A,E} = \cc{A,E} = \wpL{A,E}$.

\begin{proposition}\normalfont\label{prop:wpL-AE}
Let $A$ be a unital Banach algebra and $E$ be a Banach left $A$-module. 
If $\wpL{A}=\{0\}$ then $\wpL{A,E} = \{0\}$.
\end{proposition}
\begin{proof}
Let $f\in\wpL{A,E}$. Then, it is not difficult to show that $xf\in\wpL{A}$ for every $x\in E$ by Lemma~\ref{lemma:wpL}. Consequently, $f(E) = f(AE) = \{0\}$. Thus, $f=0$.
\end{proof}
We could weaken the assumption that $A$ is unital. The conclusion of Proposition~\ref{prop:wpL-AE} remained intact if we merely assumed that the closure of $AE$ is equal to $E$.

\begin{lemma}\normalfont\label{lem:wpL-AE-cpt}
Let $A$ be a unital Banach algebra such that $\wpL{A}=\{0\}$, and let $E$ be a Banach $A$-bimodule. 
Then, no nonzero right multiplication operator $R_a:E\to E$ is compact.
\end{lemma}
\begin{proof}
$\wpL{A,E} = \{0\}$ by Proposition~\ref{prop:wpL-AE}. 
Suppose there exists $a\neq 0$ such that $R_a$ is compact. Since $A$ is unital, there exists $f\in\dual{A}$ such that $af\neq 0$ by Hahn-Banach theorem. Therefore, $S_{af} = S_fR_a$ is compact, i.e., $af\in\wpL{A,E}$. Contradiction. 
\end{proof}


\section{The Kernel of $\wpL{A}$}\label{sec:kerwpL}

\begin{definition}\normalfont\label{def:kerwpL}
Let $A$ be a reflexive Banach algebra. 
\begin{eqnarray}\label{eq:kerwpL}
\kerwpL{A} = \wpL{A}^{\perp} = \bigcap \{\ker{f}: f\in\wpL{A}\}
\end{eqnarray}
and $N\subseteq A$ is set of all $z\in A$ such that $x_ny_n\to z$ weakly for two weakly null sequences $(x_n)$, $(y_n)$ in $A$. 
\end{definition}

\begin{lemma}\normalfont\label{lem:N}
Let $A$ be a reflexive Banach algebra, and $N$ be defined as in Definition~\ref{def:kerwpL}.
For $l=1,\dots,n$ let $(x_k^l)_{k\in\N}$ be weakly null sequences such that $x_k^1\dots x_k^n \to z$ weakly. Then, $z\in N$.
\end{lemma}
\begin{proof}
The proof is by induction. The base step $n=2$ is true by definition. Suppose the hypothesis is true for $n=m-1>2$. Let $(x_k^l)_{k\in\N}$, $l=1,\dots,m$ be weakly null sequences such that $x_k^1\dots x_k^m \to z\in A$ weakly. For brevity, let $w_k = x_k^2\dots x_k^{m}$. Since $A$ is reflexive, $(w_k)$ has a weakly convergent subsequence, say $w_{k_i}\to z_2$ weakly. By the intermediary step, $z_2\in N$. Clearly,
$\displaystyle (x_{k_i}^1(w_{k_i}-z_2) )$ is a product of two weakly null sequences, which converges weakly to $z$. Thus, $z\in N$.
\end{proof}

\begin{lemma}\normalfont\label{lem:kerwpL}
Let $A$ be a reflexive Banach algebra, and $N, \kerwpL{A}$ be defined as in Definition~\ref{def:kerwpL}. Then,
\begin{enumerate}[label=\alph*.,itemsep=0mm]
\item $AN\subseteq N$, $NA\subseteq N$. 
If $A$ is unital, then $AN=N=NA$. \label{itm:kerwpL_semigroup-ideal} 

\item $\kerwpL{A}$ is a closed two-sided ideal and $A/\kerwpL{A}$ has jwsc multiplication. 
\label{itm:kerwpL_ideal} 

\item $\wpL{A} = N^{\perp}$. 
Thus, $\kerwpL{A}$ is the closed linear span of $N$. \label{itm:kerwpL_annihilator} 

\item $\kerwpL{A} = \bigcap \{\ker{T_f}: f\in\wpL{A}\} = \bigcap \{\ker{S_f}: f\in\wpL{A}\}$ 
\label{itm:kerwpL_kerTf} 
\end{enumerate}
\end{lemma}
\begin{proof}
\ref{itm:kerwpL_semigroup-ideal} Let $a\in A$ and $z\in N$. Let $(x_n)$, $(y_n)$ be two weakly null sequences in $A$ such that $x_ny_n\to z$ weakly. Clearly, $(ax_n)$ is weakly null and $ax_ny_n\to az$ weakly. Thus, $az\in N$. Similarly, $za\in N$.

\ref{itm:kerwpL_ideal} Let $x\in A$. $fx, xf\in\wpL{A}$ for every $f\in\wpL{A}$. Thus,\\ $f(x\kerwpL{A}) = fx(\kerwpL{A})=\{0\}$ and $f(\kerwpL{A}x) = xf(\kerwpL{A}) = \{0\}$ for all $f\in\wpL{A}$. Hence, $x\kerwpL{A} \cup \kerwpL{A}x\subseteq \kerwpL{A}$.

Second, since $\kerwpL{A}$ is a two-sided ideal, then
$\ap{A/\kerwpL{A}} = \ap{A}\cap \kerwpL{A}^{\perp} = \ap{A}$ by \cite[Corollary 4.3]{UlgerDuncan92}.
Thus, $\dual{(A/\kerwpL{A})} = \ap{A/\kerwpL{A}}$.

\ref{itm:kerwpL_annihilator} $\wpL{A} \subseteq N^{\perp}$ by \cite[Lemma~3.1]{Oktay21}.
Conversely, if $f\notin\wpL{A}$, then there exist $r>0$, weakly null sequences $(x_n)$, $(y_n)$ such that $|f(x_ny_n)|\geq r$ for all $n\in\N$. By the weak compactness of the unit ball, there is a subsequence $(x_{n_k}y_{n_k})$ weakly converging to some $z\in N$. Thus, $|f(z)|\geq r$ so $f(N)\neq \{0\}$.

Second, $\kerwpL{A} = \wpL{A}^{\perp} = N^{\perp\perp}$. Thus, $\kerwpL{A}$ is the closed linear span of $N$.

\ref{itm:kerwpL_kerTf} 
$\kerwpL{A}\subseteq\ker{T_f}\subseteq\ker{f}$ for any $f\in\wpL{A}$. In fact, for any $x\in \kerwpL{A}$
$$ fx(A) = f(xA) \subseteq f(\kerwpL{A}) = \{0\}$$
Consequently, $\kerwpL{A}\subseteq \bigcap \{\ker{T_f}: f\in\wpL{A}\} \subseteq \bigcap \{\ker{f}: f\in\wpL{A}\} = \kerwpL{A}$. 
Second equality is obtained similarly.
\end{proof}

\begin{theorem}\normalfont\label{thm:fincodim-leftmaxideal}
Let $A$ be a reflexive unital Banach algebra, 
and $N$ be defined as in Definition~\ref{def:kerwpL}.
Let $I$ be a maximal (closed) left ideal 
and $P$ be a primitive ideal of $A$. Then, 
\begin{enumerate}[label=\alph*.,itemsep=0mm]
\item there exists $f\in\dual{A}$ such that $I=\ker{S_f}$ and $S_f(A)$ is closed in $\dual{A}$. 
\item $I\supseteq N$ if and only if $I$ has finite codimension.
\item $P\supseteq N$ if and only if $P$ has finite codimension. 
\end{enumerate}
\end{theorem}
\begin{proof}
a. There exists a nonzero $f\in\dual{A}$ such that $I\subseteq\ker{f}$ by Hahn-Banach theorem. 
Since $I$ is a left ideal, then $I\subseteq\ker{S_f}$. 
Thus, $I=\ker{S_f}$ by maximality.

Let $X=\overline{S_f(A)}$ in $\dual{A}$. Clearly, $X$ is reflexive and the map $\pi:A\to B(X)$ defined by $\pi(a)xf = axf$ is an irreducible representation. Thus, $X=S_f(A) \simeq A/I$, and $\pi$ is equivalent to the reduction of the left regular representation of $A$ on $A/I$, e.g., by \cite[Theorem 4.2.21]{Palmer1}. 

b. If $I\supseteq N$, then $\kerwpL{A}\subseteq\ker{f}$, so $f\in\wpL{A} = \ap{A}$. Thus, $S_f$ is a finite dimensional operator, being compact and having a closed range.

Conversely, if $I$ is finite codimensional, then there exists $f\in\dual{A}$ such that $I=\ker{f}$. Thus, $I=\ker{S_f}$ and $S_f$ is a finite dimensional operator. Hence, $f\in\ap{A}=\wpL{A}$ and so $N\subseteq\ker{S_f} = I$ by Lemma~\ref{lem:kerwpL}.

c. Every primitive ideal is of the form $\{a\in A: aA\subseteq I\}$ for some maximal left ideal $I$, and in this case, $P$ is the largest two-sided ideal contained in $I$, e.g., see \cite[Theorem 4.1.8]{Palmer1}.

If $N\subseteq P=\{a\in A: aA\subseteq I\}\subseteq I$, then the maximal left ideal $I$ has finite codimension by (b). Clearly, $P = \ker{\pi}$ where $\pi:A\to B(A/I)$ is the finite dimensional representation defined by $\pi(a)(x+I) = ax+I$. Thus, $P$ has finite codimension.

Conversely, if $P$ has finite codimension, so does $I$. Thus, $I\supseteq\kerwpL{A}\supseteq N$. Since $P$ is the largest ideal contained in $I$, then $P\supseteq \kerwpL{A}\supseteq N$. 
\end{proof}

\begin{corollary}\normalfont\label{cor:fincodim-leftmaxideal}
Let $A$ be a reflexive unital Banach algebra, and $N$ be defined as in Definition~\ref{def:kerwpL}. 
Then, the following are equivalent.
\begin{enumerate}[label=\roman*.,itemsep=0mm]
\item Every irreducible representation is finite-dimensional.
\item Every primitive ideal has finite codimension.
\item Every maximal left ideal has finite codimension.
\item $N\subseteq\jacobsonradical{A}$.
\end{enumerate}
If $A$ is also amenable, then (i-iv) is equivalent to
\begin{enumerate}
\item[v.] $A$ is finite-dimensional and semisimple.
\end{enumerate}
\end{corollary}
\begin{proof}
$(\textrm{i}\Leftrightarrow\textrm{v})$ is \cite[Corollary 2.3]{GRW92}.
\hspace{10mm}
$(\textrm{ii}\Leftrightarrow\textrm{iv}\Leftrightarrow\textrm{iii})$ by Theorem~\ref{thm:fincodim-leftmaxideal}.
\end{proof}

\begin{theorem}\normalfont\label{thm:main-raBa}
Suppose $A$ is a reflexive amenable Banach algebra. 
Let $N, \kerwpL{A}$ be defined as in Definition~\ref{def:kerwpL}. 
Then,
\begin{enumerate}[label=\alph*.,itemsep=0mm]
\item $A/\kerwpL{A}$ is semisimple and finite-dimensional. \label{itm:raBa_quotient} 
\item $\wpL{A}$ is a finite dimensional subspace of $\dual{A}$. \label{itm:raBa_wpL} 
\item The hull $\hull{N}$ is a finite set, and every $P\in\hull{N}$ is of the form $P=Az$ for some central idempotent $z\in A$. \label{itm:raBa_quotient_hull}
\item $\kerwpL{A}$ has finite codimension, and $\kerwpL{A}=A{{e}}$ for some central idempotent ${{e}}\in A$. \label{itm:raBa_M1}
\item $\kerwpL{A}$ is a reflexive amenable Banach algebra such that 
$\wpL{\kerwpL{A}} = \{0\}$. \label{itm:raBa_M2}
\end{enumerate}
\end{theorem}
\begin{proof}
\ref{itm:raBa_quotient} By \cite[Corollary 2.3]{GRW92} since $A/\kerwpL{A}$ is a reflexive amenable Banach algebra such that every primitive ideal has finite codimension.

\ref{itm:raBa_wpL} Clearly $\dual{(A/\kerwpL{A})} = \wpL{A}$ and $\dual{\kerwpL{A}} = \dual{A}/\wpL{A}$.
%
%
Since $A/\kerwpL{A}$ is finite dimensional, so is $\wpL{A}$.

\ref{itm:raBa_quotient_hull} $\hull{N}$ is finite by \eqref{itm:raBa_quotient}. Each $P\in\hull{N}$ has finite codimension, so complemented in $A$. Thus, $P$ has an identity, say $z\in P$. 
$zx = zxz = xz$ for every $x\in A$, i.e., $z$ is a central idempotent.
Thus, $P = Az$.

\ref{itm:raBa_M1} Write $\hull{N}=\{Az_1,\dots,Az_n\}$. Since $A/\kerwpL{A}$ is semisimple, $$\kerwpL{A} = \bigcap_{k=1}^n Az_k = Az_1z_2\dots z_n.$$ Clearly, ${{e}}=z_1z_2\dots z_n$ is a central idempotent.

\ref{itm:raBa_M2} Clearly $\kerwpL{A}$ is reflexive. 
Being an ideal, $\kerwpL{A}$ is amenable if and only if it has a bounded approximate identity. 
Clearly, ${{e}}(x{{e}}) = (x{{e}}){{e}} = x{{e}}$ for every $x\in A$. Thus, ${{e}}$ is the identity of $\kerwpL{A}$, and $\kerwpL{A}$ is amenable.

Third, given $w\in N$, let $(x_n),(y_n)$ be weakly null sequences in $A$ such that $x_ny_n\to w$ weakly. Then, $(x_n{{e}}),(y_n{{e}})$ are weakly null sequences in $\kerwpL{A}$ such that \hbox{$x_n{{e}}y_n{{e}} = x_ny_n{{e}} \to w{{e}} = w$} weakly. 
Thus, $f\in\wpL{\kerwpL{A}}$ if and only if $N\subseteq\ker{f}$ if and only if $f=0$. 
\end{proof}

If $\kerwpL{A}$ in Theorem~\ref{thm:main-raBa} is finite dimensional, then $\kerwpL{A}=\{0\}$, which is the case precisely when the central idempotent ${{e}}=0$. In this case, $A$ is semisimple and finite dimensional. On the other hand, if $\kerwpL{A}\neq\{0\}$, then both $\kerwpL{A}$ and $A$ are infinite dimensional.
Consequently, 
\begin{corollary}\normalfont\label{cor:conjecture}
The following two statements are equivalent.
\begin{enumerate}[label=\roman*.,itemsep=0mm]
\item Every reflexive amenable Banach algebra is finite dimensional.
\item If $A$ is a reflexive amenable Banach algebra for which $\wpL{A} = \{0\}$, then $A=\{0\}$. 
\end{enumerate}
\end{corollary}


\section{Some Implications of $\wpL{A}=\{0\}$}\label{sec:wpL0}

\begin{proposition}\normalfont\label{lem:N-wpL0}
Let $A$ be a reflexive unital Banach algebra such that \hbox{$\wpL{A}=\{0\}$}, and $N$ be defined as in Definition~\ref{def:kerwpL}.
Then, there exist $m\in\N$ and a finite subset $F=\{z_1,\dots,z_m\}\subset N$ satisfying the following.
\begin{enumerate}[label=\alph*.]
\item $1 = z_1+\dots+z_m$, thus $A=N+\dots+N$ the sum of $m$ copies of $N$.
\item The sum of $m-1$ copies of $N$ is not dense in $A$ when $m>1$. 
\item $F$ is a linearly independent set. 
\end{enumerate}
\end{proposition}
\begin{proof}
Since $\wpL{A}=\{0\}$, then $A=\overline{\spn{N}}$ by Lemma~\ref{lem:kerwpL}. Thus, there exists $u\in\spn{N}$ with $\|1-u\|<1$. $u$ is a unit, so $1\in u^{-1}\spn{N}\subset\spn{N}$ by Lemma~\ref{lem:kerwpL}.
Let 
\begin{eqnarray}\label{eq:m-copies-N}
m = \min\{|F| : F\subseteq N \hspace{2mm}\textnormal{and}\hspace{2mm}  1 = \sum_{z\in F} z \}.
\end{eqnarray}
and $F=\{z_1,\dots,z_m\}\subset N$ with $1=z_1+\dots+z_m$. Then,
$A = Az_1+\dots+Az_m\subseteq N+\dots+N$ a sum of $m$ copies of $N$.

If the sum of $m-1$ copies of $N$ were dense in $A$, then there existed a unit $u=w_1+\dots+w_{m-1}\in A$, where $w_k\in N$. The set $E=u^{-1}\{w_1,\dots,w_{m-1}\}\subset N$ satisfies $|E|<m$ and $1 = \sum_{z\in E} z$, which contradicts with \eqref{eq:m-copies-N}.

Third, for a given $z\in F$, let $L_z$ be a maximal left ideal that contains 
$\displaystyle\sum_{w\in F\backslash\{z\}} Aw$. 
$z\notin L_z$ since otherwise $\displaystyle A = \sum_{w\in F} Aw \subseteq Az + L_z = L_z$. There exists $\gamma_z\in\dual{A}$ such that $\gamma_z(z)=1$ and $L_z\subseteq\ker{\gamma_z}$. 
As a result, $F$ is a linearly independent set. 

%

\end{proof}

The following corollary of Theorem~\ref{thm:main-raBa} improves \cite[Proposition~3.4]{Runde98} by removing the approximation property from the hypotheses. 
\begin{corollary}\normalfont\label{cor:Runde98-prop3.4}
Let $A$ be a simple, reflexive, amenable Banach algebra which possess a non-zero $f\in\ap{A}$.
Then $A$ is finite-dimensional. 
\end{corollary}
\begin{proof}
Since $A$ is reflexive, then $\wpL{A} = \ap{A}\neq\{0\}$, thus $\kerwpL{A}$ is a proper ideal. Since $A$ is simple, $\kerwpL{A}=\{0\}$. Consequently, $A$ is finite-dimensional by Theorem~\ref{thm:main-raBa}.
\end{proof}
Corollary~\ref{cor:Runde98-prop3.4} could be phrased equivalently as: if $A$ is a simple, reflexive, amenable Banach algebra such that $\wpL{A}\neq\{0\}$, then $A$ is finite-dimensional.

\begin{theorem}\label{thm:minimalidempotents}\normalfont
Let $A$ be a reflexive unital Banach algebra for which \hbox{$\wpL{A}=\{0\}$.} Then, $A$ contains no minimal idempotents. 
\end{theorem}
\begin{proof}
Suppose $p\in A$ is a minimal idempotent. Since $p\notin\jacobsonradical{A}$, then there exists a maximal left ideal $L\subset A$ such that $p\notin L$.
Since $\wpL{A}=\{0\}$, then the Banach left $A$-module $E=A/L$ is infinite dimensional by Theorem~\ref{thm:fincodim-leftmaxideal}. Clearly, since $A$ is reflexive as a Banach space, so is $E$.

Let $\pi:A\to B(E)$ be the reduction of the left regular representation. 
Then, $\pi$ is an irreducible representation of $A$ (e.g. \cite{Palmer1}), $\pi(p)\neq 0$ and $\pi(p)$ is a minimal idempotent of $\pi(A)$. Thus, $\pi(A)$ contains the space of approximable operators $\mathcal{A}(E)$ by \cite[Theorem 3]{Barnes71}. On the other hand, since $\pi(A)$ is reflexive as a Banach space, so is $\mathcal{A}(E)$. Thus, $E$ must be finite dimensional. Contradiction.
\end{proof}

If $A$ is a reflexive Banach algebra with AP, then the canonical map \\
$J:\ptensor{A}{A}\to\dual{(\itensor{\dual{A}}{\dual{A}})}$ is a linear isomorphism and
\begin{eqnarray*}
\ddual{(\ptensor{A}{A})} 
= (\ptensor{A}{A}) \oplus (\itensor{\dual{A}}{\dual{A}})^{\perp}
\end{eqnarray*}
as Banach spaces by \eqref{eq:directsumtensor}. This direct sum is, indeed, a direct sum of $A$-bimodules. 
Moreover, by \eqref{eq:center-directsum}
\begin{eqnarray*}
\mathcal{Z}(A,\ddual{(\ptensor{A}{A})} )
= \mathcal{Z}(A,\ptensor{A}{A}) 
\oplus \mathcal{Z}(A,(\itensor{\dual{A}}{\dual{A}})^{\perp}).
\end{eqnarray*}

\begin{theorem}\normalfont\label{thm:raBa-wpL0}
Suppose $A$ is a reflexive unital Banach algebra and \hbox{$\wpL{A}=\{0\}$.} Then,
$\mathcal{Z}(A, \ptensor{A}{A}) \subseteq \ker{J}$. 
If $A$ also has AP, then
$\mathcal{Z}(A, \ptensor{A}{A}) = \{0\}$. 
\end{theorem}
\begin{proof}
Let $\mathbf{1}$ denote the identity of $A$.
Let $\texttt{m} = \sum_{i=1}^{\infty} x_i\otimes y_i \in \ptensor{A}{A}$ such that $a\texttt{m}=\texttt{m}a$ for all $a\in A$.
For a given $\phi\in\dual{A}$, define $T:A\to A$ by $$\displaystyle Ta = \sum_{i=1}^{\infty} \phi(y_ia)x_i,$$ and let $v=T\mathbf{1}$.
$T$ is compact, since it is the norm-limit of finite dimensional operators.
Second, for each $\psi\in\dual{A}$,
\begin{eqnarray*}
\psi(Ta) 
= \sum_{i=1}^{\infty} \psi(x_i) \phi(y_ia) 
= \psi\otimes\phi (\texttt{m}a)
= \psi\otimes\phi (a\texttt{m})
= \psi\left(\sum_{i=1}^{\infty} ax_i\phi(y_i)  \right)
= \psi(av)
\end{eqnarray*}
Thus, $T=R_v$ is a compact right multiplication operator on $A$. 
Hence, $v=0$ by Lemma~\ref{lem:wpL-AE-cpt}.

Consequently, $\psi\otimes\phi (\texttt{m}) = \sum_{i=1}^{\infty} \psi(x_i)\phi(y_i) = 0$ for each $\phi,\psi\in \dual{A}$. In particular, if $A$ has AP, we conclude that $\texttt{m}=0$.
\end{proof}

As a final note, suppose $A$ is a reflexive amenable Banach algebra $A$ with AP, and let $\kerwpL{A}$ be defined as in Definition~\ref{def:kerwpL}.
Every reflexive amenable Banach algebra $A$ has an identity, so $\kerwpL{A}$ is a reflexive unital Banach algebra such that $\wpL{\kerwpL{A}} = \{0\}$ by Theorem~\ref{thm:main-raBa}. $\kerwpL{A}$ also has AP. In fact, 
if ${{e}}$ is the identity of $\kerwpL{A}$, $K\subseteq \kerwpL{A}$ a compact subset, and $T:A\to A$ a finite dimensional operator, then ${{e}}T{{e}}:\kerwpL{A}\to\kerwpL{A}$ is a finite dimensional operator and
$$\sup_{x\in K} \|x - ({{e}}T{{e}})x\| 
= \sup_{x\in K} \|{{e}}(x - Tx)\| 
\leq \|{{e}}\| \sup_{x\in K} \|x - Tx\|.$$
Thus, $\mathcal{Z}(\kerwpL{A}, \ptensor{\kerwpL{A}}{\kerwpL{A}}) = \{0\}$ by Theorem~\ref{thm:raBa-wpL0}.
Consequently, every virtual diagonal $\texttt{M}$ of $\kerwpL{A}$ annihilates $\itensor{\dual{\kerwpL{A}}}{\dual{\kerwpL{A}}}$.

\begin{corollary}\normalfont
Let $A$ be a reflexive amenable Banach algebra with AP.
If $\mathcal{B}(A,\dual{A}) = \mathcal{K}(A,\dual{A})$, then $A$ is finite dimensional. 
\end{corollary}
\begin{proof}
Let $\kerwpL{A}\subseteq A$ be as in Definition~\ref{def:kerwpL}. Since $\kerwpL{A}$ has finite codimension in $A$ by Theorem~\ref{thm:main-raBa}, then $\mathcal{B}(\kerwpL{A},\dual{\kerwpL{A}}) = \mathcal{K}(\kerwpL{A},\dual{\kerwpL{A}})$. Also $\kerwpL{A}$ is a reflexive amenable Banach algebra with AP such that $\wpL{\kerwpL{A}} = \{0\}$ by Theorem~\ref{thm:main-raBa}, and $\mathcal{Z}(\kerwpL{A}, \ptensor{\kerwpL{A}}{\kerwpL{A}}) = \{0\}$ by by Theorem~\ref{thm:raBa-wpL0}. 
Every virtual diagonal $\texttt{M}$ of $\kerwpL{A}$ annihilates $\itensor{\dual{\kerwpL{A}}}{\dual{\kerwpL{A}}} = \mathcal{K}(\kerwpL{A},\dual{\kerwpL{A}}) = \mathcal{B}(\kerwpL{A},\dual{\kerwpL{A}}) = \dual{(\ptensor{\kerwpL{A}}{\kerwpL{A}})}$, so  $\texttt{M}=0$. Thus, $\kerwpL{A}=\{0\}$, so $A$ is finite dimensional.
\end{proof}

\bibliography{/media/onur/sda8/O/TeX/Bibtex_files/All.bib,/media/onur/sda8/O/TeX/Bibtex_files/Textbook.bib,/media/onur/sda8/O/TeX/Bibtex_files/Banach_algebra.bib,/media/onur/sda8/O/TeX/Bibtex_files/Banach_spaces.bib,/media/onur/sda8/O/TeX/Bibtex_files/jwsc/jwsc.bib}

\providecommand{\bysame}{\leavevmode\hbox to3em{\hrulefill}\thinspace}
\providecommand{\MR}{\relax\ifhmode\unskip\space\fi MR }
\providecommand{\MRhref}[2]{%
  \href{http://www.ams.org/mathscinet-getitem?mr=#1}{#2}
}
\providecommand{\href}[2]{#2}
\begin{thebibliography}{10}

\bibitem{Barnes71}
B.A. Barnes, \emph{Irreducible algebras of operators which contain a minimal
  idempotent}, Proceedings of the American Mathematical Society \textbf{30}
  (1971), 337--342.

\bibitem{BonsallDuncan73}
F.~Bonsall and J.~Duncan, \emph{Complete {N}ormed {A}lgebras}, Springer, 1973.

\bibitem{UlgerDuncan92}
J.~Duncan and A.~\"{U}lger, \emph{Almost periodic functionals on {B}anach
  algebras}, The Rocky Mountain Journal of Mathematics \textbf{22} (1992),
  no.~3, 837--848.

\bibitem{GRW92}
J.E. Gale, T.J. Ransford, and M.C. White, \emph{Weakly compact homomorphisms},
  Transactions of the American Mathematical Society \textbf{331} (1992), no.~2,
  815--824.

\bibitem{Oktay21}
O.~Oktay, \emph{Weak sequential properties of the multiplication operators on
  {B}anach algebras}, Quaestiones Mathematicae (2021).

\bibitem{Palmer1}
T.W. Palmer, \emph{Banach {A}lgebras and the {G}eneral {T}heory of *-{A}lgebras
  {I}-{II}}, Cambridge University Press, 1994.

\bibitem{Rickart60}
C.E. Rickart, \emph{General {T}heory of {B}anach {A}lgebras}, Van Nostrand,
  1960.

\bibitem{Runde98}
V.~Runde, \emph{The structure of contractible and amenable {B}anach algebras},
  Banach Algebras 97 (E.~Albrecht and M.~Mathieu, eds.), De Gruyter, 1998,
  pp.~415--430.

\bibitem{Runde02}
\bysame, \emph{Lectures on {A}menability}, Lecture Notes in Mathematics, vol.
  1774, Springer, 2002.

\bibitem{Ryan02}
R.A. Ryan, \emph{Introduction to {T}ensor {P}roducts of {B}anach {S}paces},
  Springer monographs in mathematics, Springer, New York, 2002.

\end{thebibliography}

\end{document}